\documentclass[11pt]{article}
\usepackage{amsmath,amsthm}
\usepackage[top = 1in, bottom = 1in, left = .75in, right =.75in]{geometry}
\usepackage[usenames,dvipsnames]{xcolor}
\usepackage{pgf,tikz}
\usepackage{tkz-graph}

\theoremstyle{plain}
\newtheorem{theorem}{\bf Theorem}

\theoremstyle{definition}
\newtheorem*{definition}{Definition}

\theoremstyle{remark}

\author{H\"{u}seyin Acan\\ School of Mathematical Sciences\\ Monash University\\ Melbourne, VIC 3800\\ Australia\\ 
\texttt{huseyin.acan@monash.edu}
}
\title{Forests of chord diagrams with a given number of trees}
\date{}

\begin{document}

\maketitle

\begin{abstract}
In this short note we find the number of forests of chord diagrams with a given number of trees and a given number of chords.

\vspace{.5cm}

\end{abstract}

A \emph{chord diagram} of size $n$ is a pairing of $2n$ points. One can take $2n$ points on a circle, label them with the numbers $1,\dots,2n$ in clockwise order, and join the two points in a pair with a chord for a pictorial representation. Given a chord diagram $\mathcal{C}$, the intersection graph $G_{\cal C}$ is defined as follows: the vertices of $G_{\cal C}$ correspond to the chords of $\cal C$ and two vertices in $G_{\cal C}$ are adjacent if and only if the corresponding chords cross each other in $\cal C$; see Figure~\ref{fig: CD and graph}. When $G_{\cal C}$ is a tree or forest, we call $\cal C$ a {\em tree chord diagram} or a {\em forest chord diagram}, respectively. 

\begin{figure}[h]
\centering
\begin{tabular} {c c}
\begin{tikzpicture}[scale = .8]

\coordinate (center) at (0,0);
\def\radius{1.2cm}
\draw (center) circle[radius=\radius];
\begin{scriptsize}
\foreach \x in {1,...,10}
{
\path (center) ++(180+0.1*180-\x*0.2*180:\radius) coordinate (A\x);
\fill (A\x) circle[radius=2pt] ++(180+0.1*180-\x*0.2*180:.7em) node {\x};
}
\foreach \from/\to in {A1/A8, A2/A9, A3/A5, A7/A10, A4/A6}
  \draw (\from) -- (\to);
\end{scriptsize}
\end{tikzpicture}

\hspace{2cm}

\begin{tikzpicture}
 [scale=1,auto=left,every node/.style={ inner sep=3pt}, every loop/.style={min distance=20mm}]
\coordinate (n1) at (0,0) {};
\coordinate (n2) at (1,1) {};
\coordinate (n3)  at (2,0) {};
\coordinate (n4)   at (3.5,1) {};
\coordinate (n5) at (3.5,0) {};

\fill (n1) circle[radius=3pt] node[below] {$(1,8)$};
\fill (n2) circle[radius=3pt] node[above] {$(2,9)$};
\fill (n3) circle[radius=3pt] node[below] {$(7,10)$};
\fill (n4) circle[radius=3pt] node[above] {$(3,5)$};
\fill (n5) circle[radius=3pt] node[below] {$(4,6)$};

\foreach \from/\to in {n1/n2, n2/n3, n1/n3, n4/n5}
  \draw (\from) -- (\to);
\end{tikzpicture}

\end{tabular}
\caption{A chord diagram with five chords and the corresponding intersection graph.}
\label{fig: CD and graph}
\end{figure}
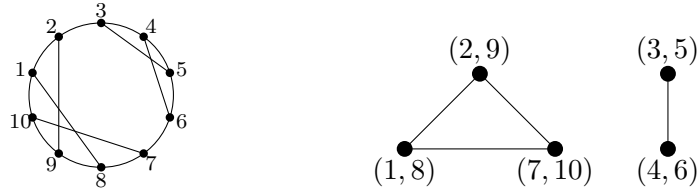

Dulucq and Penaud~\cite{DP} found the number of tree chord diagrams with $n$ chords as ${3n-3 \choose n-1}/(2n-1)$ by showing a bijection between tree chord diagrams with $n$ chords and ternary plane trees with $n-1$ internal vertices. 
In this note we find $f(n,m)$, the number of forests of chord diagrams consisting of $n$ chords and $m$ trees. We also find $r(n,m)$, the number of forest chord diagrams consisting of $n$ chords and $m$ rooted trees. In particular, we prove the following two theorems.

\begin{theorem}\label{main}
For positive integers $n$ and $m$ such that $m\le n$, we have
\begin{displaymath}
f(n,m)= 
\begin{cases}
\frac{1}{n-m}\cdot{2n \choose m-1}\cdot {3n-2m-1 \choose n-m-1}, & \text{if } 1\le m \le n-1\,;   \\
\frac{1}{n+1}\cdot{2n \choose n}, & \text{if } n=m\,.
\end{cases}
\end{displaymath}
\end{theorem}

\begin{theorem}\label{thm2}
For positive integers $n$ and $m$ such that $m\le n$, we have
\[
r(n,m) =  \frac{1}{m} {2n \choose m-1} \big( \Sigma_1+ \Sigma_2\big),
\]
where
\[ 
\Sigma_1=\sum_{k=0}^m \sum_{j=0}^{n-m-1} (-1)^k{m \choose k} {m+j-1\choose j}2^{m-k} 3^j\frac{2j+k}{n-m-j} {3n-3m+k-j-1 \choose n-m-j-1} 
\]
and
\[
\Sigma_2= \sum_{k=0}^m (-1)^k {m \choose k} {n-1 \choose n-m}  2^{m-k}3^{n-m}. 
\]
\end{theorem}

The generating function $G(x)$ of ternary trees satisfies the equation $G(x)=1+xG(x)^3$. For a positive integer $k$, let $t_k$ denote  the number of tree chord diagrams with $k$ chords and let $T(x)=\sum_{k\ge 1} t_kx^k$. By the bijection shown by Dulucq and Penaud~\cite{DP}, we have $T(x)=xG(x)$.
Hence, using the above relation for $G(x)$, we get
\begin{equation}\label{tree-gen}
\frac{T(x)}{x}=1+x\left( \frac{T(x)}{x}\right)^3.
\end{equation}

Before we start our proofs, we define non-crossing partitions and cite a result about such partitions with given parameters.

\begin{definition}
Two blocks of a partition of the set $[m]:=\{1,\dots,m\}$ cross each other if one block contains $\{a,c\}$ and the other block contains $\{b,d\}$ where $a<b<c<d$. A partition of $[m]$ is called a \textit{non-crossing partition} if no two blocks cross each other.
\end{definition}

The relevance of this definition is that the components of a chord diagram with $n$ chords determine a non-crossing partition of $[2n]$. (The components of a chord diagram $\mathcal{C}$ correspond to the components of $G_{\cal C}$.)  More precisely, each block of the partition consists of the vertices of the chords in a component. 

In order to count the forests with $m$ trees, we will consider non-crossing partitions of $[2n]$, and for each such partition we will count the trees in each block. The following theorem about the number of non-crossing partitions was proved by Kreweras~\cite{Kreweras}. (See Stanley~\cite[Exercise 5.35]{Stanley2} for a bijective proof.)  

\begin{theorem}[Kreweras]\label{non-crossing}
The number of non-crossing partitions of $[n]$ with $s_j$ blocks of size $j$ is $\frac{(n)_{k-1}}{s_1!\cdots s_n!}$, where $k=s_1+\cdots+s_n$ and $(n)_{k-1}= n(n-1)\cdots(n-k+2)$. 
\end{theorem}

\begin{proof}[{\bf Proof of Theorem~\ref{main}}]
For $n=m$, we have $n$ non-crossing chords, and the number of such configurations is given by the $n$th Catalan number, which is $(2n)!/(n!(n+1)!)$. Now assume $1\le m<n$.

For a component $H$ of a chord diagram, we define the \textit{support} of $H$ as the set of endpoints of its chords. 
Let $\{B_1,\dots, B_m\}$ be the set of blocks in a non-crossing partition of $[2n]$, where each block has an even cardinality.  The number of forest chord diagrams such that each $B_j$ is a tree is $t_{|B_1|/2}\times \cdots \times t_{|B_m|/2}$. In order to find the number of forests with $m$ trees, we need to take a sum over all non-crossing partitions of $[2n]$ with $m$ blocks of  even cardinality. 

We say that a non-crossing partition of $[2n]$ has type $(s_1,\dots,s_n)$ when there are $s_i$ blocks of size $2i$.  In that case we have $\sum_{i=1}^n is_i=2n$. By Theorem \ref{non-crossing}, the number of non-crossing partitions of $[2n]$ of type $(s_1,\cdots,s_n)$ is $(2n)!/(R!s_1!\cdots s_n!)$, where $R= 2n+1-(s_1+\cdots+s_n)$.
Hence,
\begin{equation}
\label{eq: forestsum1}	f(n,m)= \sum_{s_1,\dots,s_n} t_1^{s_1}	\cdots t_n^{s_n}\frac{(2n)!}{(2n+1-m)!s_1!\cdots s_n!}\,,	
\end{equation}
where the sum is over all nonnegative integers $s_1,\dots,s_n$ such that 
\begin{equation}\label{constraint}	
\sum_{i=1}^n s_i=m  \text{ and } \sum_{i=1}^n is_i=n .			
\end{equation}
We write \eqref{eq: forestsum1} as 
\begin{align}	
f(n,m)	&= \frac{(2n)!}{(2n+1-m)!}\sum_{s_1,\dots,s_n \atop \text{ meeting~\eqref{constraint} }} \frac{t_1^{s_1}}{s_1!}\cdots \frac{t_n^{s_n}}{s_n!}   \notag  \\
		&=   \frac{(2n)!}{(2n+1-m)!} [x^n] \frac{T(x)^m}{m!}	\notag  \\
		&= \frac{1}{m} {2n \choose m-1} [x^n] T(x)^m.
\end{align}
Finally, we need to find the coefficient of $x^n$ in the series $T(x)^m$. By Eq.~\eqref{tree-gen}, we have $z=x(z+1)^3$, where $z=z(x)=\frac{T(x)}{x}-1$. Hence,
\begin{align}\label{eq: tree coefficient}
[x^n]T(x)^m &= [x^{n-m}]\left( \frac{T(x)}{x}\right)^m 
= [x^{n-m}] (z+1)^m 	\notag	\\
&= \sum_{k=0}^{m} {m \choose k} [x^{n-m}] z^k 
= \sum_{k=0}^m { m \choose k} \frac{k}{n-m}[x^{n-m-k}](x+1)^{3(n-m)}, 
\end{align}
where the last equality follows from the Lagrange inversion formula. Thus,
\begin{align*}
f(n,m)		&= \frac{1}{m}{2n \choose m-1}\sum_{k=0}^m {m \choose k}  \frac{k}{n-m}{3(n-m) \choose n-m-k}\\
&= \frac{1}{n-m}{2n \choose m-1}\sum_{k=1}^m {m-1 \choose k-1} {3(n-m) \choose n-m-1-(k-1)} \\
&=  \frac{1}{n-m}{2n \choose m-1}{3n-2m-1\choose n-m-1}		
\end{align*}
as desired.
\end{proof}

\begin{proof}[{\bf Proof of Theorem~\ref{thm2}}]
Analogous to Eq.~\eqref{eq: forestsum1}, we have
\begin{equation}  \label{eq: r(n,m)1}
r(n,m)= \sum_{s_1,\dots,s_n} (1\cdot t_1)^{s_1}(2\cdot t_2)^{s_2} \cdots (n\cdot t_n)^{s_n}  \frac{(2n)!} {(2n+1-m)!s_1!\cdots s_n!},
\end{equation}
since in this case we need to choose a root for each of the trees in a forest. The sum is over all nonnegative integer tuples $(s_1,\dots,s_n)$ satisfying~\eqref{constraint}.
We rewrite~\eqref{eq: r(n,m)1} as
\begin{equation}  \label{eq: r(n,m)2}
r(n,m)= \frac{(2n)!}{(2n+1-m)!} \sum_{s_1,\dots,s_n} \frac{(1\cdot t_1)^{s_1}}{s_1!}\cdots \frac{(n\cdot t_n)^{s_n}}{s_n!}. 
\end{equation}
Let $B(n,m)$ be this sum  without the coefficient $(2n)!/(2n+1-m)!$, i.e.,
\[	 
B(n,m)= \sum_{s_1,\dots,s_n} \frac{(1\cdot t_1)^{s_1}}{s_1!}\cdots \frac{(n\cdot t_n)^{s_n}}{s_n!}.			
\]
We have
\begin{align*}
\sum_{n,m \ge 1}B(n,m)x^my^n  
&= \sum_{n,m}\sum_{s_1,\dots,s_n} \frac{(1\cdot t_1)^{s_1}}{s_1!}\cdots \frac{(n\cdot t_n)^{s_n}}{s_n!} x^my^n	\\
&=\prod_{j\ge 1} \sum_{s\ge 0} \frac{(jt_jxy^j)^s}{s!} 		\\
&= \exp \left( x\sum_{j\ge 1} jt_jy^j\right)  = e^{xR(y)},
\end{align*}
where $R(y)=\sum_{n\ge 1} nt_ny^n$. Thus, $B(n,m)$ is the coefficient of $x^my^n$ in the function $e^{xR(y)}$, which is the same as the coefficient of $y^n$ in $R(y)^m/m!$. 
Hence, using~\eqref{eq: r(n,m)2}
\begin{equation} 	\label{eq: r(n,m)}
r(n,m) = \frac{1}{m} {2n \choose m-1} [x^n]R(x)^m.	
\end{equation}
Finally, we need to find the coefficient of $x^n$ in $R(x)^m$. First note that $R(x)=~xT'(x)$, where $T(x)$ is the generating function for unrooted trees. By~\eqref{tree-gen}, $T(x)$ satisfies 
\[		xT(x)= x^2+T(x)^3,		\]
from which it follows
\[		R(x)=xT'(x)= x\cdot \frac{2x-T(x)}{x-3(T(x))^2}.		\]
Then,
\[	[x^n]R(x)^m= [x^{n-m}]\left( \frac{2x-T(x)}{x-3T(x)^2}\right) ^m.		\]
Writing now $z=T(x)$,
{\allowdisplaybreaks
\begin{align}\label{eq: [x^n]R(x)^m}
&[x^n]R(x)^m = [x^{n-m}] \Big( 2-\frac{z}{x}\Big)^m  \Big( 1-\frac{3z^2}{x} \Big)^{-m}			\notag		\\
=&[x^{n-m}] \Bigg( \sum_{k=0}^m {m \choose k} 2^{m-k}(-1)^k \left( \frac{z}{x}\right)^k \times \sum_{j=0}^{\infty}  {-m \choose j}(-3)^j\left( \frac{z^2}{x}\right)^j 		\Bigg) 	\notag	\\
=&\sum_{k=0}^m \sum_{j=0}^{n-m} {m \choose k} {-m \choose j}2^{m-k} (-3)^j(-1)^k\, [x^{n-m+k+j}] z^{2j+k} 	\notag	 \\
=&\sum_{k=0}^m \sum_{j=0}^{n-m-1} (-1)^k{m \choose k} {m+j-1\choose j}2^{m-k} 3^j\frac{2j+k}{n-m-j} {3n-3m+k-j-1 \choose n-m-j-1}	\notag		\\
&\quad + \sum_{k=0}^m (-1)^k {m \choose k} {n-1 \choose n-m}  2^{m-k}3^{n-m},
\end{align}}
where we use 
\[	
[x^b]z^{a}=
\begin{cases}
\frac{a}{b-a}{3b-2a-1 \choose b-a-1}, & \text{ if } a<b; \\
1, & \text{ if } a=b
\end{cases}
\]
in the last equality above (cf.\  Eq.~\eqref{eq: tree coefficient}).  Combining~\eqref{eq: r(n,m)} and~\eqref{eq: [x^n]R(x)^m}, we get
\begin{equation*}
r(n,m) =  \frac{1}{m} {2n \choose m-1} \big( \Sigma_1+ \Sigma_2\big),
\end{equation*}
where
\[ 
\Sigma_1=\sum_{k=0}^m \sum_{j=0}^{n-m-1} (-1)^k{m \choose k} {m+j-1\choose j}2^{m-k} 3^j\frac{2j+k}{n-m-j} {3n-3m+k-j-1 \choose n-m-j-1} 
\]
and
\[
\Sigma_2= \sum_{k=0}^m (-1)^k {m \choose k} {n-1 \choose n-m}  2^{m-k}3^{n-m},
\]
which finishes the proof.
\end{proof}

\end{document}